\documentclass[11pt,reqno]{amsart}
\usepackage[usenames]{color}
\usepackage{amsmath, amssymb, latexsym, multicol, amsthm, color, enumitem,comment,bbm}
\usepackage[all]{xy}
\newtheorem{theorem}{Theorem}
\newtheorem{lemma}[theorem]{Lemma}

\newtheorem{proposition}[theorem]{Proposition}
\newtheorem{conjecture}[theorem]{Conjecture}

\theoremstyle{remark}

\newtheorem*{remarks}{Remarks}

\newcommand{\Z}{\mathbb{Z}}
\newcommand{\Q}{\mathbb{Q}}
\newcommand{\N}{\mathbb{N}}

\newcommand{\R}{\mathbb{R}}
\newcommand{\C}{\mathbb{C}}

\renewcommand{\H}{\mathbb{H}}

\newcommand{\gen}{\mathbb{G}}
\newcommand{\genus}{\operatorname{gen}}
\newcommand{\ord}{\text{ord}}
\newcommand{\spn}{\mathbb{S}} 
\newcommand{\spin}{\operatorname{spn}}
\newcommand{\cls}{\operatorname{cls}}
\newcommand{\lan}{\langle}
\newcommand{\ran}{\rangle}

\newcommand{\SL}{\operatorname{SL}}

\numberwithin{equation}{section}
\numberwithin{theorem}{section}

\title[Ternary sums of polygonal numbers]{Almost universal ternary sums of polygonal numbers}

\author{Anna Haensch}
\address{Anna Haensch, Department of Mathematics and Computer Science, Duquesne University, Pittsburgh, PA 15282, USA.}
\email{haenscha@duq.edu}

\author{Ben Kane}
\address{Ben Kane, Department of Mathematics, University of Hong Kong, Pokfulam, Hong Kong.}
\email{bkane@hku.hk}

\thanks{The research of the second author was supported by grant project numbers 27300314, 17302515, and 17316416 of the Research Grants Council. }

\begin{document}

\begin{abstract}
For a natural number $m$, generalized $m$-gonal numbers are those numbers of the form $p_m(x)=\frac{(m-2)x^2-(m-4)x}{2}$ with $x\in \mathbb Z$.  In this paper we establish conditions on $m$ for which the ternary sum $p_m(x)+p_m(y)+p_m(z)$ is almost universal. 
\end{abstract}
\date{\today}
\maketitle

\section{Introduction} 

For a natural number $m$, the $x$-th generalized $m$-gonal number is given by $p_m(x)=\frac{(m-2)x^2-(m-4)x}{2}$ where $x\in \mathbb Z$.  In 1638, Fermat claimed that every natural number may be written as the sum of at most $3$ triangular numbers, $4$ squares, $5$ pentagonal numbers, and in general $m$ $m$-gonal numbers.  Lagrange proved the four squares theorem (the $m=4$ case) in 1770, Gauss proved the the triangular number theorem (the $m=3$ case) in 1796, and Cauchy proved the full claim in 1813 \cite{Cauchy}.  Guy \cite{Guy} later investigated the minimal number $r_m\in\N$ chosen such that every natural number may be written as the sum of $r_m$ generalized $m$-gonal numbers.  For $m\geq 8$, Guy noted that an elementary argument shows that one needs $m-4$ generalized $m$-gonal numbers to represent $m-4$, so $m-4\leq r_m\leq m$.  However, he pointed out that for large enough $n\in\N$, one could likely represent $n$ with significantly fewer generalized $m$-gonal numbers.  In this paper, we investigate for which $m$ every sufficiently large $n\in\N$ is the sum of three $m$-gonal numbers.  That is to say, we study representations of natural numbers by the ternary sum
\[
P_m(x,y,z):=p_m(x)+p_m(y)+p_m(z),
\]
and we ask for which $m$ the form $P_m$ is almost universal; a form is called \begin{it}almost universal\end{it} if it represents all but finitely many natural numbers.  In other words, we would like to determine the set of $m$ for which the set
\[
\mathcal{S}_m:=\!\left\{n\in\N: \not\exists (x,y,z)\in \Z^3\text{ with } P_m(x,y,z)=n\right\}
\]
is finite.  The set $\mathcal{S}_m$ is those positive integers which are not represented by $P_m$, and we call $P_m$ \begin{it}almost universal\end{it} if $\mathcal{S}_m$ is finite.  
\begin{theorem}\label{thm:main}
If $m\not\equiv 2\pmod{3}$ and $4\nmid m$, then $P_m$ is almost universal.  
\end{theorem}
\begin{remarks}
\noindent

\noindent
\begin{enumerate}[leftmargin=*,label={\rm(\arabic*)}]
\item
Theorem \ref{thm:main} states that for $m\not\equiv 2\pmod{3}$ and $4\nmid m$, every sufficiently large natural number may be written as the sum of at most three generalized $m$-gonal numbers.  However, its proof relies on Siegel's ineffective bound \cite{Siegel} for the class numbers of imaginary quadratic orders, so the result does not give an explicit bound $n_m$ such that every $n>n_m$ may be written as the sum of three generalized $m$-gonal numbers.
\item
Questions of almost universality have recently been studied by a number of authors, but in a slightly different way.  In most cases, $(m_1,m_2,m_3)\in\N_{\geq 3}^3$ has been fixed and authors investigated representations by weighted sums of the type
\[
ap_{m_1}(x)+bp_{m_2}(y)+cp_{m_3}(z).
\]
In particular, authors worked on the classification of $a,b,c$ for which the above form is almost universal; on the contrary, in Theorem \ref{thm:main} we fix $a$, $b$, and $c$ and vary $m$.  In \cite{KS08}, for example, a classification of such $a,b,c$ was given for $m_1=m_2=4$ and $m_3=3$.  In the case of a weighted sum of triangular numbers, a partial answer was given in \cite{KS08}, and the characterization of such almost universal sums was completed by Chan and Oh in \cite{CO09}.  The most general result to date appears in \cite{H14}, where a characterization of almost universal weighted sums of $m$-gonal numbers is given for $m-2=2p$ with $p$ an odd prime.  In this last case, the results in \cite{H14} imply that $P_m$ is not almost universal (note that $m\equiv 0\pmod{4}$, so this is partially complementary to the result in Theorem \ref{thm:main}).
\end{enumerate}
\end{remarks}

It turns out that the restrictions $m\not\equiv 2\pmod{3}$ and $4\nmid m$ are both necessary in Theorem \ref{thm:main}, but are of a very different nature.  If $4\mid m$, then there is a \begin{it}local obstruction\end{it} to $P_m$ being almost universal, i.e., there is an entire congruence class $A\N_0+B\subseteq \mathcal{S}_m$ because it is not even represented modulo $A$.   Details of these local obstructions may be found in Lemma \ref{lem:local}. 

The restriction $m\not\equiv 2\pmod{3}$ (with $4\nmid m$) is much more delicate, and a seemingly deep connection between the analytic and algebraic theory lies beneath this case.  In this case, there are no local obstructions, but $\mathcal{S}_m$ is not necessarily finite.  To get a better understanding of the set $\mathcal{S}_m$, for $m$ even we define 
\[
\mathcal{S}_{m,3}^{\operatorname{e}}:=\!\left\{n\in \mathcal{S}_m: \exists r\in\Z\text{ with }2(m-2)n+3\left(\frac{m-4}{2}\right)^2 = 3r^2\right\}
\]
and for $m$ odd we define
\[
\mathcal{S}_{m,3}^{\operatorname{o}}:=\!\left\{n\in \mathcal{S}_m: \exists r\in\Z\text{ with }8(m-2)n+3(m-4)^2 = 3r^2\right\}.
\]
We next see that if $m\equiv 2\pmod{3}$, then most of the exceptional set $\mathcal{S}_m$ is contained in $\mathcal{S}_{m,3}^{\operatorname{e}}$ if $m\equiv 2\pmod{4}$ and contained in $\mathcal{S}_{m,3}^{\operatorname{o}}$ if $m$ is odd.
\begin{theorem}\label{thm:2mod3}
\noindent

\noindent
\begin{enumerate}[leftmargin=*,label={\rm(\arabic*)}]
\item
If $m\equiv 2\pmod{12}$, then $\mathcal{S}_m\setminus\mathcal{S}_{m,3}^{\operatorname{e}}$ is finite.  
\item
If $m\equiv 2\pmod{3}$ and $m$ is odd, then $\mathcal{S}_m\setminus\mathcal{S}_{m,3}^{\operatorname{o}}$ is finite.
\end{enumerate}
\end{theorem}

The proof of Theorem \ref{thm:main} mainly uses the analytic approach and relies on ineffective bounds of class numbers.  However, this approach fails to gain any control in determining the sets $\mathcal{S}_{m,3}^{\operatorname{e}}$ and $\mathcal{S}_{m,3}^{\operatorname{o}}$.  One is hence motivated to blend the two approaches together in order to investigate these sets.  From the algebraic point of view, representations of an integer $n$ by $P_m$ is equivalent to representations of a related integer by a lattice coset $L+\nu$, where $L=L_{(m)}$ and $\nu=\nu_{(m)}$ are completely determined by $m$ (for the precise formulation of $L$ and $\nu$ see Section \ref{sec:prelim}).  To explain why combining the algebraic and analytic theories may be beneficial, we recall an important interplay between the analytic and algebraic theories which occurs when $P_m$ is replaced with the quadratic form $Q$ on the positive-definite ternary lattice $L$, called the \begin{it}norm\end{it} on $L$, which we later emulate.  To understand the link, for such a lattice $L$, let $\mathcal{L}$ denote the primitive elements of $L$ (those which are not non-zero integral multiples of other elements of $L$) and set
\[
\mathcal{S}_{L}:=\{n\in\N : \not\exists \alpha\in \mathcal{L}\text{ with }Q(\alpha)=n\}.
\]
Since $L$ is a positive-definite lattice, there will always be local obstructions at an odd number of finite primes, but our main consideration is those $n\in\mathcal{S}_{L}$ which are locally represented, which we refer to as \begin{it}locally admissible\end{it}.  Moreover, there are finitely many primes $p$ (known as anisotropic primes) for which every $\nu\in L$ with $Q(\nu)$ highly divisible by $p$ is necessarily imprimitive (i.e., $\nu=p\nu'$ for some $\nu'\in L$).  Therefore, if $\ord_p(n)$ is large, we immediately conclude that $n\in \mathcal{S}_L$, so we restrict $\ord_p(n)$.  Using the analytic theory, one can show that for the special case where $L$ is a lattice, the subset of $n\in \mathcal{S}_L$ which are both locally admissible and have bounded $p$-adic order (by a specific constant depending on $p$) at all anisotropic primes $p$ is finite outside of finitely many square classes $t_1\Z^2,\dots,t_{\ell}\Z^2$.  This follows from a result of Duke and Schulze-Pillot in \cite{DS-P90}.  The behavior inside these square classes is explained via the spinor norm map in the algebraic theory; this occurs by realizing $t_j\Z^2$ as a spinor exceptional square class; the {\em primitive spinor exceptions} for the genus of $L$ are those integers which are primitively represented by some but not all of the spinor genera in the genus of $L$. These primitive spinor exceptions are determined by Earnest, Hsia, and Hung in \cite{EH94}.  With additional investigation one can use these results to determine the existence of infinite subsets of admissible elements of $\mathcal{S}_{L}$; from this one can determine that if the subset of admissible elements of $\mathcal{S}_L$ with bounded divisibility at the anisotropic primes is infinite, then there is at least one spinor exceptional square class.  

Returning to our case $P_m$, one would expect a similar theory of spinor exceptional square classes to emerge if one could link the algebraic and analytic approaches.  It is revealing that for $4\nmid m$ the only possibly infinite part of $n\in \mathcal{S}_m$ occurs when $An+B$ is within the square class $3\Z^2$, hinting at a synthesis between the approaches yet to be investigated.  In order to state a conjectural link in our case, we next recall the link between the two approaches in a little more detail.

The main synthesis between the analytic and algebraic theories goes through the Siegel--Weil (mass) formula.  For a lattice $L_0$ in a positive-definite space, let $\gen(L_0)$ be a set of representatives of the classes in the genus $\genus(L_0)$ of $L_0$.  One version of the Siegel--Weil formula states that 
\[
\mathcal{E}_{\genus\!\left(L_0\right)}:=\frac{1}{\sum_{L\in\gen\!\left(L_0\right)} \omega_L^{-1}}\sum_{L\in \gen\!\left(L_0\right)} \frac{\Theta_L}{\omega_L}
\]
is a certain Eisenstein series.  Here $\Theta_L$ is the theta function associated to $L$ (i.e., the generating function for the elements of $L$ of a given norm; see \eqref{eqn:thetadef}) and $\omega_L$ is the number of \begin{it}automorphs\end{it} of $L$ (i.e., the number of linear \begin{it}isometries\end{it} from $L$ to itself; these are invertible linear maps on the vector space $\Q L$ which fix $L$ and preserve the associated quadratic form $Q$).  If $L_0$ has rank $3$, then if one instead takes the associated sum over a set $\spn(L_0)$ of representatives of the classes of lattices 
in the spinor genus $\spin(L_0)$, one obtains
\begin{equation}\label{eqn:spnquadratic}
\frac{1}{\sum_{L\in\spn(L_0)} \omega_L^{-1}}\sum_{L\in \spn(L_0)} \frac{\Theta_L}{\omega_L}= \mathcal{E}_{\genus(L_0)} + \mathcal{U}_{\spin(L_0)},
\end{equation}
where $\mathcal{U}_{\spin(L_0)}$ is a linear combination of unary theta functions \cite{SP,SP2}.  The Fourier coefficients of $\mathcal{U}_{\spin(L_0)}$ count the excess or deficiency of the weighted average of the number of representations by the spinor genus of $L_0$ when compared with the weighted average of the number of representations by the genus, giving a direct connection back to the algebraic theory, the spinor norm map, and spinor exceptions.  The key observation which makes \eqref{eqn:spnquadratic} useful is that the left-hand side is a weighted average of modular forms all of whose coefficients are non-negative.  Hence if the $n$-th coefficient of this sum is zero, then the $n$-th coefficient of each summand must also be zero, and these coefficients count the number of representations of $n$.  On the other hand, the functions appearing on the right-hand side of \eqref{eqn:spnquadratic} are special types of modular forms whose Fourier coefficients may be explicitly computed.

After rewriting the question about representations by $P_m$ as a question about representations by a particular lattice coset $L_{(m)}+\nu_{(m)}$ (defined in \eqref{eqn:Ldef} and \eqref{eqn:nudef}), one would expect such a theory to hold in our case as well.  Indeed, the Siegel--Weil formula for the genus of every lattice coset $L+\nu$ was proven by van der Blij \cite{vanderBlij} and then later independently by Shimura \cite{Shimura2004}, who showed that  
\[
\Theta_{\genus(L+\nu)}=\mathcal{E}_{\genus(L+\nu)}:=\frac{1}{\sum_{M+\nu'\in \gen(L+\nu)} \omega_{M+\nu'}^{-1}}\sum_{M+\nu'\in \gen(L+\nu)} \frac{\Theta_{M+\nu'}}{\omega_{M+\nu'}}
\]
is an Eisenstein series, where $\omega_{M+\nu'}$ is the number of automorphs of the lattice coset and $\gen(L+\nu)$ denotes a complete set of representatives of the classes in the genus of $L+\nu$.  Kneser further showed in \cite{Kneser2} how this formula for the genus of lattice cosets follows by investigating the Haar measure on the orthogonal group, but we do not take that perspective in this paper.  We conjecture that the expected link holds in the same way for spinor genera of lattice cosets. 
\begin{conjecture}\label{conj:SiegelWeil}
We have 
\[
\Theta_{\spin(L+\nu)}:=\frac{1}{\sum_{M+\nu'\in \spn(L+\nu)} \omega_{M+\nu'}^{-1}}\sum_{M+\nu'\in \spn(L+\nu)} \frac{\Theta_{M+\nu'}}{\omega_{M+\nu'}}=\mathcal{E}_{\genus(L+\nu)} + \mathcal{U}_{\spin(L+\nu)},
\]
where $\mathcal{U}_{\spin(L+\nu)}$ is a linear combination of unary theta functions and $\spn(L+\nu)$ denotes a set of representatives of the classes in the spinor genus of $L+\nu$.
\end{conjecture}

Conjecture \ref{conj:SiegelWeil} is useful in two different ways.  Firstly, it shows that the number of representations by the spinor genus is usually the same as the number of representations by the genus, and secondly it is useful for showing that certain integers in the support of the unary theta functions are \underline{not} represented by a given lattice coset.  To better understand the utility of Conjecture \ref{conj:SiegelWeil} and to motivate why we believe it to be true, we return to $P_m$.  In particular, for $m=14$, a finite calculation yields the following.
\begin{proposition}\label{prop:SiegelWeil}
The theta function $\Theta_{\spin(L_{(14)}+\nu_{(14)})}$ satisfies Conjecture \ref{conj:SiegelWeil}.
\end{proposition}

As stated above, one of the main advantages of Proposition \ref{prop:SiegelWeil} is that one can use it to show that the Fourier coefficients of $\Theta_{\spin(L_{14}+\nu_{14})}$ usually agree with those of 
\[
\mathcal{E}_m:=\mathcal{E}_{\genus(L_{(m)}+\nu_{(m)})}.
\]
However, we specifically use Proposition \ref{prop:SiegelWeil} to investigate the coefficients supported by the unary theta functions to prove that infinitely many coefficients of $\Theta_{14}$ in these square classes vanish, where 
\[
\Theta_m:=\Theta_{L_{(m)}+\nu_{(m)}}.
\]
We then build off of this to use Proposition \ref{prop:SiegelWeil} to prove that $P_{m}$ is not almost universal for every $m\equiv 2\pmod{12}$.

\begin{theorem}\label{thm:P14}
For every $m\equiv 2\pmod{12}$, the form $P_{m}$ is not almost universal.
\end{theorem}
\begin{remarks}
\noindent

\noindent
\begin{enumerate}[leftmargin=*,label={\rm(\arabic*)}]
\item
For any given lattice coset $L+\nu$, one can check Conjecture \ref{conj:SiegelWeil} with a (possibly long) finite calculation. To show that Conjecture \ref{conj:SiegelWeil} is true for all lattice cosets, one would need to develop the algebraic theory further to determine spinor exceptions (resp. primitive spinor exceptions) for lattice cosets, proving a theorem analogous to Schulze-Pillot's results in \cite{SP80} (resp. Earnest, Hsia, and Hung's results in \cite{EH94}).
\item
It is natural to ask whether one expects the forms $P_m$ to be almost universal in the case that $m\equiv 2\pmod{3}$ is odd.  Guy showed in \cite{Guy} that $P_5$ is not only almost universal, but indeed universal.  Computer calculations indicate that $P_{11}$ is also almost universal.  In order to prove that any given $P_m$ in this family is almost universal, it suffices to decompose the associated theta function into an Eisenstein series, a linear combination of unary theta functions, and a cusp form which is orthogonal to unary theta functions.  If the contribution from unary theta functions is trivial, then form will be almost universal.  Following Conjecture \ref{conj:SiegelWeil}, one expects the unary theta function contribution to directly appear from the theta function associated to the spinor genus.  
\end{enumerate}
\end{remarks}

The paper is organized as follows.  We first give some preliminary definitions and known results in Section \ref{sec:prelim}.  In Section \ref{sec:algebraic}, we use algebraic methods to establish the local behavior of $P_m$.  In Section \ref{sec:analytic}, we give a proof Theorem \ref{thm:main} using analytic methods.  We then finally blend the two approaches together in Section \ref{sec:link} in order to prove Proposition \ref{prop:SiegelWeil} and Theorem \ref{thm:P14}.
\section*{Acknowledgements}

The authors thank Wai Kiu Chan for helpful discussion and Rainer Schulze-Pillot for pointing out the work of van der Blij \cite{vanderBlij} and Kneser \cite{Kneser2} related to the Siegel--Weil formula for shifted lattices after seeing a preliminary version of this paper.  The authors also thank the anonymous referee for many corrections and also suggestions that improved the exposition of the paper. 

\section{Preliminaries}\label{sec:prelim}
In this section, we introduce the necessary objects used in the algebraic proofs.

\subsection{Setup for the algebraic approach:  Lattice theory}\label{sec:prelimlattice}

For the algebraic approach, we adopt the language of quadratic spaces and lattices as set forth in \cite{OM}.  If $L$ is a lattice and $A$ is the Gram matrix for $L$ with respect to some basis, we write $L\cong A$.  When $A$ is a diagonal matrix with entries $a_1,...,a_n$ on the diagonal, then $A$ is written as $\lan a_1,...,a_n\ran$.  For a lattice $L$, we let $V$ denote the underlying quadratic space; that is, $V=\Q L$. In this case, we say that $L$ is a lattice on the quadratic space $V$.  For a lattice $L$ we define the localization of $L$ by $L_p=L\otimes_\Z\Z_p$, where now $L_p$ is a $\Z_p$-lattice on $V_p:=V\otimes_{\Q} \Q_p$.

Given a lattice $L$ and a vector $\nu\in V$, we have the lattice coset $L+\nu$.  If we define a lattice $M=L+\Z\nu$, then $L+\nu$ can be regarded as coset inside the lattice quotient $M/L$. Elements in $L+\nu$ are simply vectors of the form $\nu+x$, where $x\in L$.

We are considering representations of an integer $n$ by the sum $P_m$, which upon completing the square, is seen to be equivalent to the condition that 
\[
\ell_n:=\begin{cases}
3\left(\frac{(m-4)}{2}\right)^2+2(m-2)n & \text{ for $m$ even}\\
3\left(m-4\right)^2+8(m-2)n & \text{ for $m$ odd}
\end{cases}
\] 
is represented by the lattice coset $L+\nu$, where $L=L_{(m)}$ is defined as the $\Z$-lattice 
\begin{equation}\label{eqn:Ldef}
\begin{cases}
\lan(m-2)^2,(m-2)^2,(m-2)^2\ran & \text{ for $m$ even}\\
\lan4(m-2)^2,4(m-2)^2,4(m-2)^2\ran & \text{ for $m$ odd},\\
\end{cases}
\end{equation}
 in the orthogonal basis $\{e_1,e_2,e_3\}$, and 
\begin{equation}\label{eqn:nudef}
\nu=\nu_{(m)}:=
\frac{(m-4)}{2(m-2)}\!\left(e_1+e_2+e_3\right). 
\end{equation}
To prove Theorem \ref{thm:main}, we need to show that all but finitely many $\ell_n$ are represented by the lattice coset $L+\nu$.  

In order to approach this problem from the algebraic side, we need to develop some algebraic notion of the class, spinor genus, and genus of a lattice coset.  Following the definitions that originally appear in \cite{CO13} the {\em class} of $L+\nu$ is defined as 
\begin{equation}\label{eqn:clscosetdef}
\cls(L+\nu):=\text{the orbit of $L+\nu$ under the action of $SO(V)$,}
\end{equation}
the {\em spinor genus} of $L+\nu$ as 
\begin{equation}\label{eqn:spncosetdef}
\spin(L+\nu):=\text{the orbit of $L+\nu$ under the action of $SO(V)O'_\mathbb{A}(V)$,}
\end{equation}
and the {\em genus} of the lattice coset $L+\nu$ by 
\begin{equation}\label{eqn:gencosetdef}
\genus(L+\nu):=\text{ the orbit of $L+\nu$ under the action of $SO_\mathbb{A}(V)$,}
\end{equation}
and where $O'_\mathbb{A}(V)$ denotes the adeles of the kernel of the spinor norm map, $\theta:SO(V)\rightarrow \mathbb Q^\times/{\mathbb Q^\times}^2$ as defined in \cite[\S 55]{OM}. Note that what we refer to as the genus (resp. spinor genus or  class) above is often called the \begin{it}proper genus\end{it} (resp. proper spinor genus or proper class), and is commonly denoted with a superscript $+$; e.g., the proper genus is written $\genus^+(L+\nu)$, while the (non-proper) genus (resp. spinor genus and class) are usually defined with the corresponding special orthogonal groups (e.g., $SO_{\mathbb{A}}(V)$) replaced by the orthogonal groups (e.g. $O_{\mathbb{A}}(V)$).  Although the genus and proper genus are always equal in the cases of lattices (see \cite[\S102 A]{OM}) this is not always true for lattice cosets.  In particular, if $O(L_p+\nu)$ does not contain an improper isometry (an element of the orthogonal group with determinant $-1$) at some finite prime $p$, then $\text{gen}^+(L+\nu)\subsetneq\text{gen}(L+\nu)$. For an example of this phenomenon, we direct the reader to \cite[Example 4.5]{CO13}.  In our case we are guaranteed  that $O(L_p+\nu)$ contains a symmetry at every prime $p$.  This is obvious at primes $p$ not dividing $2(m-2)$ since in this case $L_p+\nu=L_p$ is just a diagonal lattice.  At other primes we can take the symmetry $\tau_{e_1-e_2}$, which switches the basis elements $e_1$ and $e_2$ while fixing $e_3$.   Therefore the $\text{gen}(L+\nu)=\text{gen}^+(L+\nu)$.  Since it will not make a difference in this setting, we choose to define the class, genus, and spinor genus above in terms of the special orthogonal group so that our notation matches that given in \cite{Xu}, which will be helpful to us in what follows.  We let $\gen(L+\nu)$ (resp. $\spn(L+\nu)$) denote a set of representatives of the classes in $\genus(L+\nu)$ (resp. $\spin(L+\nu)$). For any further unexplained notation, the reader is directed to \cite{OM}.

The general strategy will be to show first that there are no local obstructions, i.e. that $Q(\nu)+2(m-2)n$ is represented by the $\genus(L+\nu)$.  Next we will determine conditions under which the spinor genus and genus coincide.  An essential ingredient here will be to count the number of spinor genera in the genus of a lattice coset.  For this we turn to a formula given by Xu in \cite{Xu}, counting the number of spinor genera in $\genus(L+\nu)$,  
\begin{eqnarray}\label{spincount}
[J_\mathbb{Q}:\Q^\times\prod_{p\in \Omega}\theta(SO(L_p+\nu))]
\end{eqnarray}
where $J_\Q$ is the set of ideles of $\Q$ and $\Omega$ is the set of primes in $\Q$ and $SO(L_p+\nu)$ is the stabilizer of $L_p+\nu$ in $SO(V_p)$. One easily checks that (cf. \cite{CO13}) 
\begin{equation}\label{orthog}
SO(L_p+\nu)=\{\sigma\in SO(V_p):\sigma(L_p)=L_p \text{ and }\sigma(\nu)\equiv \nu\mod L_p\}.
\end{equation}
In Theorem \ref{prop:spngen} we will explicitly compute the image of $SO(L_p+\nu)$ under the spinor norm map and count the number of spinor genera. 

\subsection{Setup for the analytic approach:  Modular forms theory}
We require some results about (classical holomorphic) modular forms.  
\subsubsection{Basic definitions}
Let $\H$ denote the \begin{it}upper half-plane\end{it}, i.e., those $\tau=u+iv\in \C$ with $u\in\R$ and $v>0$.  The matrices $\gamma=\left(\begin{smallmatrix} a&b\\ c&d\end{smallmatrix}\right)\in\SL_2(\Z)$ (the space of two-by-two integral matrices with determinant $1$) act on $\H$ via \begin{it}fractional linear transformations\end{it} $\gamma \tau:=\frac{a\tau+b}{c\tau+d}$.  For 
\[
j(\gamma,\tau):=c\tau+d,
\]
a \begin{it}multiplier system\end{it} for a subgroup $\Gamma\subseteq \SL_2(\Z)$ and \begin{it}weight\end{it} $r\in \R$ is a function $\nu:\Gamma\mapsto \C$ such that for all $\gamma,M\in\Gamma$ (cf. \cite[(2a.4)]{Pe1})
\[
\nu(M \gamma) j(M\gamma,\tau)^r = \nu(M)j(M,\gamma \tau)^r \nu(\gamma)j(\gamma,\tau)^r.
\]
The \begin{it}slash operator\end{it} $|_{r,\nu}$ of weight $r$ and multiplier system $\nu$ is then 
\[
f|_{r,\nu}\gamma (\tau):=\nu(\gamma)^{-1} j(\gamma,\tau)^{-r} f(\gamma \tau).
\]
A \begin{it}(holomorphic) modular form\end{it} of weight $r\in\R$ and multiplier system $\nu$ for $\Gamma$  is a function $f:\H\to\C$ satisfying the following criteria:
\noindent

\noindent
\begin{enumerate}[leftmargin=*]
\item
The function $f$ is holomorphic on $\H$.
\item
For every $\gamma\in\Gamma$, we have 
\begin{equation}\label{eqn:modularity}
f|_{r,\nu}\gamma= f.
\end{equation}
\item
The function $f$ is bounded towards every \begin{it}cusp\end{it} (i.e., those elements of $\Gamma\backslash(\Q\cup\{i\infty\})$).  This means that at each cusp $\varrho$ of $\Gamma\backslash \H$, the function $f_{\varrho}(\tau):=f|_{r,\nu}\gamma_{\varrho}(\tau)$ is bounded as $v\to \infty$, where $\gamma_{\varrho}\in \SL_2(\Z)$ sends $i\infty$ to $\varrho$.  
\end{enumerate}
Furthermore, if $f$ vanishes at every cusp (i.e., $\lim_{\tau\to i\infty} f_{\varrho}(\tau)=0$), then we call $f$ a \begin{it}cusp form\end{it}.  

\subsubsection{Half-integral weight forms}
We are particularly interested in the case where $r=k+1/2$ with $k\in\N_0$ and 
\[
\Gamma=\Gamma_1(M):=\left\{ \left(\begin{matrix}a&b\\ c&d\end{matrix}\right)\in\SL_2(\Z): M\mid c, a\equiv d\equiv 1\pmod{M}\right\}
\]
for some $M\in\N$ divisible by $4$.  The multiplier system we are particularly interested in is given in \cite[Proposition 2.1]{Shimura}, although we do not need the explicit form of the multiplier for this paper.

If $T^N\in \Gamma$ with $T:=\left(\begin{smallmatrix} 1&1\\ 0 &1\end{smallmatrix}\right)$, then by \eqref{eqn:modularity} we have $f(\tau+N)=f(\tau)$, and hence $f$ has a Fourier expansion ($a_{f}(n)\in\C$)
\begin{equation}\label{eqn:fexp}
f(\tau)=\sum_{n\geq 0} a_{f}(n) e^{\frac{2\pi i n \tau}{N}}.
\end{equation}
The restriction $n\geq 0$ follows from the fact that $f$ is bounded as $\tau\to i\infty$.  One commonly sets $q:=e^{2\pi i \tau}$ and associates the above expansion with the corresponding formal power series, using them interchangeably unless explicit analytic properties of the function $f$ are required.

\subsubsection{Theta functions for quadratic polynomials}
In \cite[(2.0)]{Shimura}, Shimura defined theta functions associated to lattice cosets $L+\nu$ (for a lattice $L$ of rank $n$) and polynomials $P$ on lattice points.  Namely, he defined
\[
\Theta_{L+\nu,P}(\tau):=\sum_{\boldsymbol{x}\in L+\nu} P(\boldsymbol{x}) q^{Q(\boldsymbol{x})},
\]
where $Q$ is the quadratic map on the associated quadratic space.  We omit $P$ when it is trivial.  In this case, we may write $r_{L+\nu}(\ell)$ for the number of elements in $L+\nu$ of norm $\ell$ and we get
\begin{equation}\label{eqn:thetadef}
\Theta_{L+\nu}(\tau)=\sum_{\ell\geq 0} r_{L+\nu}(\ell) q^{\ell}.
\end{equation}
Shimura then showed (see \cite[Proposition 2.1]{Shimura}) that $\Theta_{L+\nu}$ is a modular form of weight $n/2$ for $\Gamma_1(4N^2)$ (for some $N$ which depends on $L$ and $\nu$) and a particular multiplier. Note that we have taken $\tau\mapsto 2N\tau$ in Shimura's definition.  To show the modularity properties, for $\gamma=\left(\begin{smallmatrix}a&b\\ c&d\end{smallmatrix}\right)\in \Gamma_1(4N^2)$, we compute
\[
2N\gamma(\tau)=2N\frac{a\tau+b}{c\tau+d} = \frac{a(2N\tau) + 2Nb}{\frac{c}{2N} (2N\tau)+d} = \left(\begin{matrix} a & 2Nb\\ \frac{c}{2N} & d\end{matrix}\right) (2N\tau).
\]
Since $\gamma\in \Gamma_{1}(4N^2)$, we have 
\[
\left(\begin{matrix} a & 2Nb\\ \frac{c}{2N} & d\end{matrix}\right)\in \Gamma(2N):=\left\{ \gamma= \left(\begin{matrix}a&b\\ c&d\end{matrix}\right)\in\SL_2(\Z): \gamma\equiv I_{2}\pmod {N}\right\}\subset \Gamma_{1}(2N),
\]
so we may then use \cite[Proposition 2.1]{Shimura}.  Specifically, the multiplier is the same multiplier as $\Theta^3$, where $\Theta(\tau):=\sum_{n\in\Z} q^{n^2}$ is the classical Jacobi theta function.  

We only require the associated polynomial in one case.  Namely, for $n=1$ and $P(x)=x$, we require the \begin{it}unary theta functions\end{it}  (see \cite[(2.0)]{Shimura} with $N\mapsto 2N^2/t$, $P(m)=m$, $A=(2N^2/t)$, and $h\mapsto 2Nh$, multiplied by $(2N)^{-1}$)
\begin{equation}\label{eqn:unarydef}
\vartheta_{h,t}(\tau)=\vartheta_{h,t,N}(\tau):=\sum_{\substack{r\in\Z\\ r\equiv h\pmod{\frac{N}{t}}}} r q^{t r^2},
\end{equation}
where $h$ may be chosen modulo $N/t$ and $t$ is a squarefree divisor of $N$.  These are weight $3/2$ modular forms on $\Gamma_1(4N^2)$ the same multliplier system as $\Theta_{L+\nu}$.

\section{Algebraic Approach}\label{sec:algebraic}

As seen in Section \ref{sec:prelimlattice}, a natural number $n$ is represented by $P_m(x,y,z)$ if and only if $\ell_n$ is represented by the lattice coset $L+\nu$.  In this section, we check for local obstructions; i.e., we check whether $\ell_n$ may be represented by $L_p+\nu$ for every prime $p$.

\begin{lemma}\label{lem:local}
If $m\equiv 0\pmod 4$ then $P_m$ is not almost universal.  
\end{lemma}

\begin{proof}
When $m=2p+2$ for an odd prime $p$, then the claim follows immediately from \cite[Theorem 7]{H14}.  Otherwise it can be easily verified that if $m\equiv 0\pmod 4$ then $P_m(x,y,z)$ always fails to represent an entire square class modulo 8, and is therefore not almost universal. Specifically, if $m\equiv 4\pmod{8}$, then $P_m(x,y,z)$ does not represent any integer congruent to $-1$ modulo $8$, while if $m\equiv 0\pmod{8}$, then $P_m(x,y,z)$ does not represent any integer congruent to $4$ modulo $8$. 

\end{proof}

 
In order for $P_m$ to be almost universal, a necessary condition is that every integer $\ell_n$ is represented by $\genus(L+\nu)$.  Since it will be helpful in much of what follows, we define the ternary lattice $M:=L+\Z\nu$ and note that this lattice has a basis $\{\nu,e_1,e_2\}$.   We will also define $T:=\{p\text{ prime }:p\mid (m-2)\}$.



\begin{lemma}\label{lem:Lp}
For any odd prime $p\notin T$, we have $M_p=L_p=L_p+\nu$.
\end{lemma}
\begin{proof}
This follows immediately from the fact that $\nu\in L_p$.  
\end{proof}
\begin{lemma}\label{genrep1}
If $m\not\equiv 0\pmod{4}$, then $\ell_n$ is represented by $L_p+\nu$ for every prime $p$.
\end{lemma}
\begin{proof}
For odd $p\notin T$, Lemma \ref{lem:Lp} implies that $L_p+\nu=L_p$ and since $L_p$ is unimodular, it represents every integer in $\Z_p$ (cf. \cite[92:1b]{OM}).

For odd $p\in T$ and $p\neq 3$, $\ell_n$ is a unit in $\Z_p$, since 
 \[
 Q(\nu)=\begin{cases}
 3\left(\frac{m-4}{2}\right)^2 & \text{ when $m-2$ is even}\\
 3(m-4)^2 & \text{ when $m-2$ is odd}
 \end{cases}
 \]
 is never divisible by $p$.   Therefore, since $Q(\nu)$ is represented by $M_p$, it follows from the local square theorem that $\ell_n$ is represented by $M_p$ for every choice of $n$.  Suppose that $\ell_n$ is represented by an arbitrary coset $L_p+t\nu$ of $L_p$ in $M_p$, where $t\in \{0,..,p^k-1\}$.  Then 
\[
Q(\nu)\equiv Q(\omega+t\nu)\equiv t^2Q(\nu)\pmod {p^k}
\]
for $\omega\in L_p$.  Consequently, $t=\pm 1$, since the multiplicative group $(\Z/p^k\Z)^\times$ contains at most one subgroup of order 2.  Therefore $\ell_n$ is represented by the coset $L_p+\nu$.  

Finally, when $p=2$, we will proceed by showing that in fact every integer in $\Z_2$ can be written as an $m$-gonal number when $m-2\equiv 0\mod 4$.    We may suppose that $\ord_2(m-2)=k+1$ where $k>0$.  Therefore, $(m-2)=2^{k+1}\epsilon$ and $(m-4)=2\gamma$ where $\epsilon,\gamma\in \Z_2^\times$.  Then an integer $n$ can be written as an $m$-gonal number precisely when there exists $x\in \Z_2$ such that 
\begin{equation}\label{T2univ}
n=\frac{(m-2)x^2-(m-4)x}{2}=2^{k}\epsilon x^2-\gamma x.
\end{equation}
The $x$ in \eqref{T2univ} (in the algebraic closure of $\Z_2$) is given by 
\begin{equation}\label{solveforx}
x=\frac{\gamma\pm\sqrt{\gamma^2-4(2^k\epsilon)(-n)}}{2^{k+1}\epsilon}
=\frac{1\pm\sqrt{1+2^{k+2}\alpha n}}{2^{k+1}\beta}
\end{equation}
where $\alpha=\epsilon/\gamma^2$ and $\beta=\epsilon/\gamma$.  By the local square theorem, we know that $1+2^{k+2}\alpha n$ is the square of a unit in $\Z_2$. Therefore, 
\[
1+2^{k+2}\alpha n=(1+2^s \delta)^2=1+2^{s+1}\delta+2^{2s}\delta^2=1+2^{s+1}(\delta+2^{s-1}\delta), 
\] 
where $s>0$ and $\delta\in \Z_2^\times$, and 
\[
x=\frac{1\pm \sqrt{1+2^{k+2}\alpha n}}{2^{k+1}\beta}=\frac{1\pm\sqrt{(1+2^s \delta)^2}}{2^{k+1}\beta}=\frac{1\pm(1+2^{s}\delta)}{2^{k+1}\beta}.
\]  
When $s>1$, since $\left| 2^{k+2}\alpha n\right|_2=\left| 2^{s+1}(\delta+2^{s-1}\delta)\right|_2$, it follows that $k+2+r=s+1$ where $r=\ord_2(n)$.  Therefore, 
\[
x=\frac{1-(1+2^s\delta)}{2^{k+1}\beta}=\frac{2^s\delta}{2^{k+1}\beta}=\frac{2^r\delta}{\beta}\in \Z_2,
\]
since $s=k+1+r$.  On the other hand, when $s=1$, then $k+2+r=2+\ord_2(1+\delta)$, and therefore, 
\[
x=\frac{1+\sqrt{(1+2 \delta)^2}}{2^{k+1}\beta}=\frac{1+(1+2\delta)}{2^{k+1}\beta}=\frac{2+2\delta}{2^{k+1}\beta}=\frac{1+\delta}{2^{k}\beta}=\frac{2^r\delta}{\beta}\in \Z_2,
\]
since $k+r=\ord_2(1+\delta)$.  Therefore, since every $2$-adic integer can be expressed as an $m$-gonal number, it follows that every $\ell_n$ is represented by the coset $L_2+\nu$.  

When $m-2$ is odd a similar argument follows, by letting $(m-2)=\epsilon$ and $(m-4)=\gamma$ where $\epsilon,\gamma\in \Z_2^\times$ and then simply replacing equation (\ref{T2univ}) with
\[
2n=(m-2)x^2-(m-4)x= \epsilon x^2-\gamma x.
\]
Hence equation (\ref{solveforx}) becomes
\[
x=\frac{\gamma\pm \sqrt{\gamma^2-4(\epsilon)(-2n)}}{2\epsilon}=\frac{\gamma\pm \sqrt{\gamma^2+8\alpha n}}{2\beta}
\]
where$\alpha =\epsilon/\gamma^2$ and $\beta=\epsilon/\gamma$, and the result follows as above. 

\end{proof}

Having established the local conditions, we next calculate the number of spinor genera for $L+\nu$.  Recall from (\ref{spincount}), the number of spinor genera in the genus of the coset is given by 
\begin{eqnarray*}
[J_\mathbb{Q}:\Q^\times\prod_{p\in \Omega}\theta(SO(L_p+\nu))]
\end{eqnarray*}
where $J_\Q$ is the set of ideles of $\Q$ and $\Omega$ is the set of primes in $\Q$. From this formula, we see that much like in the case of lattices, $\Z_p^\times\subseteq \theta(SO(L_p+\nu))$ for every prime $p$ is sufficient, though certainly not necessary to guarantee that $\genus(L+\nu)$ and $\spin(L+\nu)$ coincide.  

\begin{proposition}\label{prop:spngen}
\noindent
\begin{enumerate}
\item If $m\equiv 2\pmod 4$ and $m\not\equiv 2\pmod {12}$, then $\spin(L+\nu)=\genus(L+\nu)$.
\item For $m\equiv 2\mod 12$, there are two spinor genera in the genus of $L+\nu$.
\end{enumerate}
\end{proposition}

\begin{proof}
(1)  For primes $p\not\in T$, it is immediate that $\Z_p^\times\subseteq \theta(SO(L_p+\nu))$ since $L_p+\nu=L_p\cong\lan 1,1,1\ran$.  For primes $p\in T$, we have $(m-2)=p^k\epsilon$ and $\frac{m-4}{2}=\gamma$ where $k\geq 1$ and $\epsilon,\gamma\in \Z_p^\times$.  Then, in the basis $\{\nu,e_1,e_2\}$ we have 
\[
M_p\cong\begin{bmatrix}
3\gamma^2 & p^k\epsilon\gamma & p^k\epsilon\gamma \\
p^k\epsilon\gamma  & p^{2k}\epsilon^2 & 0\\
p^k\epsilon\gamma  & 0 & p^{2k}\epsilon^2,
\end{bmatrix}
\]
and by a change of basis to $\{\nu,p^k\epsilon\nu-3\gamma e_1,p^k\epsilon\nu-3\gamma e_2\}$ we obtain 
\[
M_p\cong \lan3\gamma^2\ran\perp 3p^{2k}\epsilon^2\gamma^2\begin{bmatrix}
6 & -3\\-3 & 6
\end{bmatrix}.
\]
From this we clearly see that $\Z_p[\nu]$ splits $M_p$ as an orthogonal summand; in other words, $M_p$ is the orthogonal sum $M_p =\Z_p[\nu]\perp K_p$, where 
\[
K_p\cong 9p^{2k}\epsilon^2\gamma^2\begin{bmatrix}
2 & -1\\-1 & 2
\end{bmatrix}
\]
is a binary modular lattice.  From here it follows immediately from \cite[Satz 3]{K56} that $\Z_p^\times\subseteq \theta(SO(K_p))$ for odd prime $p$. When $p=2$ the result follows from \cite[Lemma 1]{H75}.  On the other hand, setting
\[
\omega:=\frac{p^k\epsilon}{\gamma}\nu=e_1+e_2+e_3
\] 
the set of vectors $\{\omega,\omega-3e_1,\omega-3e_2\}$ form a basis for $L_p$, and in this basis we obtain 
\[
L_p= \Z_p[\omega]\perp K_p.
\]
Any isometry $SO(K_p)$ can be extended to an isometry $\sigma\in SO(M_p)$ which simultaneously satisfies $\sigma(\nu)=\nu$ and $\sigma(L_p)=L_p$, and therefore $\sigma(L_p+\nu)=L_p+\nu$.  Hence $\sigma\in SO(L_p+\nu)$, from which we may conclude that $\theta(SO(K_p))\subseteq \theta(SO(L_p+\nu))$ and hence $\Z_p^\times\subseteq \theta(SO(L_p+\nu))$.  

Now for any $\vec{x}=(x_p)\in J_\Q$ we know that $x_p$ is a unit at almost every prime.  Therefore, multiplying by a suitable element of $a\in \Q^\times$ we can assume $a\vec{x}=(ax_p)$ is a unit at every prime.  Moreover, since for the infinite prime $\theta(L_{\infty}+\nu)=\theta(SO(V_{\infty}))=\R^{\times^2}$, we only need to chose $a$ to have the same sign as $x_\infty$.  Chosen in this way, $a\vec{x}$ is an element in the restricted product. 

(2)  When $m\equiv 2\mod 12$, then $L\cong \lan (m-2)^2,(m-2)^2,(m-2)^2\ran$ in the basis $\{e_1,e_2,e_3\}$ and $\nu=\frac{m-4}{2(m-2)}[e_1+e_2+e_3]$.  For primes away from $T$, we once again know that $L_p+\nu=L_p\cong\lan 1,1,1\ran$, and hence $\Z_p^\times\subseteq \theta(SO(L_p+\nu))$.   Moreover, for primes $p\neq 3$ in $T$, the argument from above is still sufficient to show that $\Z_p^\times\subseteq \theta(SO(L_p+\nu))$.   When $p=2$ we make one further observation, namely that in this case $\Z_2^\times= \theta(SO(L_2+\nu))$.  If $\sigma\in SO(L_2+\nu)$ then we know $\sigma(\nu)=\nu+x$ for $x\in L_2$, hence 
\[
Q(\nu)=Q(\sigma(\nu))=Q(\nu+x)=Q(\nu)+Q(x)+2B(\nu,x),
\]
by a simple congruence argument we see that no nontrivial $x$ can satisfy this equality.  Therefore the only isometries of $L_2+\nu$ are those fixing $\nu$, and hence are precisely the isometries of $K_2$ described above.  In particular, it follows from \cite[Lemma 1]{H75} that $\Z_2^\times{\Q_2^\times}^2=\theta(SO(K_2))=\theta(SO(L_2+\nu))$.

When $p=3$, then we consider the generalized lattice $M/L$, as defined in \cite{T97}, which has the orthogonal group 
\[
O(M_3/L_3)=\{\sigma\in O(V_3):\sigma(x)\in x+L_3 \text{ for all }x\in M_3\},
\]  
also defined in \cite{T97}.  An isometry $\sigma$ is in $O(M_3/L_3)$ precisely when $\sigma(L_3)=L_3$ and $\sigma(\nu)\equiv \nu\mod L_3$.  Therefore, from (\ref{orthog}), we see that $O(M_3/L_3)=O(L_3+\nu)$ and hence $SO(M_3/L_3)=SO(L_3+\nu)$.  However, from \cite[Theorem 2]{T97} we know that $\theta(SO(M_3/L_3))$, and hence $\theta(SO(L_3+\nu))$ is generated by pairs of symmetries coming from $O(M_3/L_3)$.  If $\tau$ is a symmetry in $O(M_3/L_3)$, then there is some $\omega=e_1x_1+e_2x_2+e_3x_3\in L_3$ such that 
\[
\tau(y)=\tau_\omega(y)=y-\frac{2B(\omega,y)}{Q(\omega)}\omega
\]
for every $y\in L_3$.  We may assume that $x_1,x_2,x_3\in \Z_3$, and without loss of generality, that $x_1\in \Z_3^\times$.  But now 
\[
\tau_\omega(e_1)=e_1-\frac{2B(\omega,e_1)}{Q(\omega)}\omega=e_1-\frac{2\cdot (m-2)^2}{(m-2)^2(x_1^2+x_2^2+x_3^2)}\omega=e_1-\frac{2}{(x_1^2+x_2^2+x_3^2)}\omega\in L_3
\]
and hence $x^2+y^2+z^2\not \equiv 0\mod 3$.  This means that at least one of $x_2$ and $x_3$ is not a unit, without loss of generality, say $x_3\not\in \Z_3^\times$.  On the other hand, 
\[
\tau_\omega(\nu)=\nu-\frac{2B(\omega,\nu)}{Q(\omega)}\omega=\nu-\frac{(m-4)(m-2) (x_1+x_2+x_3)}{(m-2)^2(x_1^2+x_2^2+x_3^2)}\omega=\nu-\frac{(m-4)(x_1+x_2+x_3)}{(m-2)(x_1^2+x_2^2+x_3^2)}\omega
\] 
and since $\tau_\omega(\nu)\equiv\nu\mod L_3$ it must follow that $x_1+x_2+x_3\equiv 0\mod 3$.  Therefore the only possibility is that $x_2\in \Z_3^\times$ and $x_1\not\equiv x_2\mod 3$.  Therefore, $(m-2)^{-2}Q(\omega)\equiv 2\mod 3$, and consequently $\theta(SO(M_3/L_3))$, and hence $\theta(SO(L_3+\nu))$, contains no nontrivial elements.  That is, $2\not \in \theta(SO(L_3+\nu))$.  
Finally, we will show that the number of spinor genera in the genus of $L+\nu$, in this case, is equal to 2.  In order to show that 
\[
\left[J_\mathbb{Q}:\Q^\times\prod_{p\in \Omega}\theta(SO(L_p+\nu))\right]=2,
\]
we prove that the principal idele $1$ and the idele $\iota$, given by 
\[
\iota_p:=\begin{cases} 1 & \text{if } p\neq 3,\\ 2 &\text{if }p=3,\end{cases}
\]
are inequivalent and the cosets $[1]$ and $[\iota]$ are a full set of representatives of the quotient space. 

For any $\vec{x}=(x_p)\in J_\Q$, we know that $x_p$ is a unit for almost every $p$.  Multiplying by a suitable element $a$ in $\Q^\times$ (where $a$ has the same sign as $x_\infty$) if necessary, we may assume that $ax_p$ is a unit at every prime $p$ (including the infinite prime).  Since $\Z_p^{\times}\subseteq \theta(SO(L_p+\nu))$ for $p\neq 3$, the coset of $\vec{x}$ is completely determined by the congruence class of $ax_3$.  If $ax_3\equiv 1\mod 3$ then $\vec{x}\in [1]$ and if $ax_3\equiv 2\mod 3$ then $\vec{x}\in [\iota]$.

\end{proof}

Although Proposition \ref{prop:spngen} (1) doesn't directly lead to a proof of Theorem \ref{thm:main}, it gives a strong expectation for the results given in Theorem \ref{thm:main}.  Namely, there is a result of Duke and Schulze-Pillot \cite{DS-P90} which used the analytic theory to obtain the conclusion in the case of lattices that every sufficiently large integer primitively represented by the spinor genus is also represented by the lattice.  Since the lattice and genus coincide by Proposition \ref{prop:spngen} (1), one may expect a result similar to Duke and Schulze-Pillot's to imply Theorem \ref{thm:main}.  Since no analogous theorem has yet been developed, we turn to a trick in the analytic theory to prove Theorem \ref{thm:main}.

\section{Analytic approach}\label{sec:analytic}

In this section, we use the analytic proof to show Theorem \ref{thm:main}, Theorem \ref{thm:2mod3} (2), and the first statement of Theorem \ref{thm:2mod3} (1).  These are rewritten in the following theorem. 

\begin{theorem}\label{thm:m2mod4}
Suppose that $m\not\equiv 0\pmod{4}$.  Then we have the following. 

\noindent
\begin{enumerate}[leftmargin=*,label={\rm(\arabic*)}]
\item
If $m\not\equiv 2\pmod{3}$, then every sufficiently large $n$ may be represented in the form 
$$
n=p_m(x)+p_m(y)+p_m(z)
$$
for some $x,y,z\in\Z$.  That is to say, $P_m$ is almost universal.
\item
If $m\equiv 2\pmod{12}$, then every sufficiently large $n\notin \mathcal{S}_{m,3}^{\operatorname{e}}$ may be represented in the form 
$$
n=p_m(x)+p_m(y)+p_m(z)
$$
for some $x,y,z\in\Z$. 
\item
If $m\equiv 2\pmod{3}$ is odd, then every sufficiently large $n\notin \mathcal{S}_{m,3}^{\operatorname{o}}$ may be represented in the form 
$$
n=p_m(x)+p_m(y)+p_m(z)
$$
for some $x,y,z\in\Z$.
\end{enumerate}
\end{theorem}
\begin{proof}
We split the proof into four pieces.  First we assume that $m\equiv 2\pmod{4}$ and then split depending on the congruence class of $m$ modulo $3$, and we will later assume that $m$ is odd. 

By completing the square, a solution to the given representation is equivalent to a solution to
$$
2(m-2)n+3\!\left(\frac{m-4}{2}\right)^2= \!\left((m-2)x-\frac{m-4}{2}\right)^2+\!\left((m-2)y-\frac{m-4}{2}\right)^2+\!\left((m-2)z-\frac{m-4}{2}\right)^2.
$$
We set $N:=(m-2)$ and $\ell=\ell_n:=2(m-2)n+3\left(\frac{m-4}{2}\right)^2$.  Denoting by $r_{m}(\ell)$ the number of such solutions (with $r_m(\ell):=0$ if $\ell$ is not in the correct congruence class), we hence consider the generating function 
\begin{equation}
\Theta_m(\tau):=\sum_{n\geq 0} r_{m}\!\left(\ell_n\right) q^{\ell_n},
\end{equation}
where $q:=e^{2\pi i \tau}$.  The function $\Theta_m$ is a theta series for a lattice coset.  Since the gram matrix $A=N^2\cdot I_3$ associated to the lattice is diagonal with even entries, \cite[Proposition 2.1]{Shimura} (with $P(m)=1$, $\tau\mapsto 2N\tau$, and $h=((m-4)/2,(m-4)/2,(m-4)/2)^T$) implies that $\Theta_m$ is a weight $3/2$ modular form on $\Gamma_1(4N^2)$ with some multiplier.  As usual, we decompose
\begin{equation}\label{eqn:decompose}
\Theta_m(\tau)=\mathcal{E}_m(\tau)+\mathcal{U}_m(\tau)+f_m(\tau),
\end{equation}
where $\mathcal{E}_m$ is in the space spanned by Eisenstein series, $\mathcal{U}_m$ is in the space spanned by unary theta functions, and $f_m$ is a cusp form which is orthogonal to unary theta functions.  Of course, each term in the decomposition is modular of weight $3/2$ on $\Gamma_1(4N^2)$ with the same multiplier.  

We now follow an argument of Duke and Schulze-Pillot \cite{DS-P90}, who proved that sufficiently large integers are primitively represented by quadratic forms if and only if they are primitively represented by the spinor genus of the quadratic form (i.e., they investigated the coefficients of theta series with $\Theta_m$ replaced with the theta series for a lattice).  By work of Duke \cite{Duke}, the Fourier coefficients of $f_m$ grow at most like $\ell^{3/7+\varepsilon}$, while the coefficients of $\mathcal{E}_m$ are certain class numbers and by Siegel's (ineffective) bound \cite{Siegel} for class numbers, they grow like $\gg \ell^{1/2-\varepsilon}$ for any $\ell$ supported on the coefficients of $\mathcal{E}_m$ and which are primitively represented by the genus.  The requirement that the representations are primitive comes from the fact that there are certain primes $p$ for which the $p^r\ell$-th coefficients of Eisenstein series do not grow as a function of $r$; this phenomenon is explained on the algebraic side in the case of lattices by realizing $p$ as an anisotropic prime.  When the power of such primes is bounded, the coefficients of $\mathcal{E}_{m}$ grow faster than the coefficients of $f_m$.   We may hence disregard $f_m$ completely whenever the $\ell$-th coefficient is not supported in $\mathcal{U}_m$ and the power of bad primes $p$ dividing $\ell$ is bounded.  However, Shimura \cite{Shimura2004} used the Siegel--Weil formula for inhomogeneous quadratic forms (i.e., quadratic polynomials) to show that $\mathcal{E}_m$ is the weighted average of representations by members of the genus of the lattice coset and simultaneously the product of the local densities.  For $p\mid m-2$ with $p\neq 3$, the power of $p$ dividing $\ell$ for $p\mid m-2$ is bounded from above by the congruence conditions.  Similarly, if $\ord_3(m-2)> 1$, then $\ord_3(\ell)$ is bounded.  In the special case that $\ord_3(m-2)=1$, we note that $X^2+Y^2+Z^2\equiv 0\pmod{3}$ implies that either $3$ divides each of $X$, $Y$, and $Z$ or none of them.  Hence, in this case, the local density at $3$ for representations of $\ell$ by the lattice coset equals the local density at $3$ corresponding to the number of primitive representations by the lattice.  Since $3$ is not an anisotropic prime for the lattice $\lan 1,1,1\ran$, we conclude that the local densities grow as expected.  Since $L_p+\nu=L_p$ for $p \nmid m-2$ by Lemma \ref{lem:Lp}, the coefficients of $\mathcal{E}_m$ grow $\gg \ell^{\frac{1}{2}-\varepsilon}$ whenever they are represented (due to the fact that $\ell$ is primitively represented by the lattice).  Hence the congruence conditions for $\mathcal{E}_m$ are equivalent to checking that the integer is represented locally, or in other words that the genus represents the given integer, which follows immediately from Lemma \ref{genrep1}.  

We claim furthermore that $\mathcal{U}_m$ is identically zero, from which the claim will follow.  We may decompose $\mathcal{U}_m(\tau)$ into a linear combination of finitely many unary theta functions (defined in \eqref{eqn:unarydef}). The goal now is to determine the possible $\vartheta_{h,t}$ in the decomposition with non-zero coefficient.  We do so by restricting the possible choices of $t$ with congruence conditions on $tr^2$ implied by the definition of $\Theta_m$.  The $\ell$ upon which the coefficients of $q^{\ell}$ in $\vartheta_{h,t}$ are supported must satisfy
$$
\ell =tr^2 \equiv 0\pmod{t}.
$$
However, the coefficients of the unary theta function are supported on the same integers as the original theta series $\Theta_m$, since the Eisenstein series are also supported on these coefficients and there would hence otherwise be integers $\ell$ upon which the coefficient of $\Theta_m$ is negative, contradicting the fact that it is a generating function for the non-negative integers $r_{m}(\ell)$.  Therefore, we conclude that
$$
2(m-2)n+3\left(\frac{m-4}{2}\right)^2=\ell\equiv 0\pmod{t}.
$$
Since $2(m-2)$ is even and $3\left((m-4)/2\right)^2$ is odd, we conclude that $t$ must be odd.  The congruence then becomes 
$$
3\left(\frac{m-4}{2}\right)^2\equiv 0\pmod{t},
$$
where $t$ is some divisor of (the odd part of) $m-2$.  Rewritten, this implies that 
$$
t\; \Big| \left(m-2,3\left(\frac{m-4}{2}\right)^2\right).
$$
Now note that if $p\mid m-2$ and $p\mid m-4$, then 
$$
p\mid m-2-(m-4)=2\implies p=2.
$$
Thus, since $(m-4)/2$ is odd, 
$$
\left(m-2,3\left(\frac{m-4}{2}\right)^2\right)=(m-2,3).
$$
We now split the proof into two cases to prove (1) and (2).

To prove (1) for $m$ even, we assume that $m\not\equiv 2\pmod{3}$, so 
$$
t\mid (m-2,3)=1.
$$
We conclude that $t=1$.  However, since $m-2$ is even and $(m-4)/2$ is odd, we also have
$$
\ell =\left((m-2)x+\frac{m-4}{2}\right)^2+\left((m-2)y+\frac{m-4}{2}\right)^2+\left((m-2)z+\frac{m-4}{2}\right)^2\equiv 3\pmod{8}.
$$
In particular, $\ell=tr^2\equiv 3\pmod{8}$ implies that $t=1$ is impossible.  Since there are no possible choices of $t$, we conclude that $\mathcal{U}_m=0$.  This gives the first claim in the case $m\equiv 2\pmod{4}$.

To prove (2), we assume that $m\equiv 2\pmod{3}$ (i.e., $m\equiv 2\pmod{12}$), so that 
$$
t\mid (m-2,3)=3
$$
implies that $t=1$ or $t=3$.  The case $t=1$ is again impossible by the congruence condition modulo $8$ considered in part (1).  It follows that $\mathcal{U}_m$ is a linear combination of forms all of which have $t=3$.  Hence every $n$ suffiently large for which the corresponding $\ell$ is not of the form $3r^2$ must be represented as the sum of three $m$-gonal numbers.

We now consider the case $m$ odd.  In this case, 
\[
n\!=\!p_m(x)+p_m(y)+p_m(z)
\]
is equivalent to 
\begin{multline}\label{eqn:modd}
8n(m-2)+3(m-4)^2 = \left(2(m-2)x+(m-4)\right)^2 + \left(2(m-2)y+(m-4)\right)^2 +\\ 
+\left(2(m-2)z+(m-4)\right)^2. 
\end{multline}

Letting $R_{m}(\ell)$ be the number of solutions to \eqref{eqn:modd} with $\ell=\ell_n:=8n(m-2)+3(m-4)^2$, by \cite[Proposition 2.1]{Shimura} we see that 
\[
\Theta_m'(\tau):=\sum_{n\geq 0} R_{m}\!\left(\ell_n\right) q^{\frac{\ell_n}{4N}}
\]
is a weight $3/2$ modular form on $\Gamma_1(4N^2)$ with some multiplier.  Here $N=(m-2)$, as in the $m\equiv 2\pmod{4}$ case above.   Firstly, Lemma \ref{genrep1} implies that $\ell_n$ is represented locally, or equivalently, by the genus of the lattice coset. We conclude that the relevant coefficients of the Eisenstein series $\mathcal{E}_m'$ in the decomposition \eqref{eqn:decompose} are positive and it remains to again determine the unary theta functions $\vartheta_{h,t,2N}$ which may occur in the decomposition \eqref{eqn:decompose}.  Arguing as before, we have $\ell \equiv 0\pmod{t}$ and $t\mid 2N$.  However, since $m$ is odd, $m-4$ is also odd, so $\ell\equiv 3\pmod{8}$.  It follows that $t$ is odd, and hence $t\mid N$.  We conclude that 
\[
t\mid \left(m-2, 3(m-4)^2\right).
\]  
Since $(m-2,m-4)=1$, we conclude that $t\mid (m-2,3)$.  

Now we complete the proof of (1) when $m$ is odd. If $m\not\equiv 2\pmod{3}$, then necessarily $t=1$.  However, $\ell\equiv 3\pmod{8}$ implies that $\ell$ is not a square, and hence $t=1$ is impossible.  

To prove (3), we assume $m$ is odd and $m\equiv 2\mod 3$.  Then, since $t$ is a divisor of $1$ or $3$, we conclude that $t=1$ or $t=3$.  However, $t=1$ is again impossible because $\ell_n\equiv 3\pmod{8}$ (as defined before the definitions of $\Theta_m$ and $\Theta_m'$ above) for every $n\in\N_0$.  Therefore we have $t=3$ and the only possible exceptions are in the square class $3\Z^2$.  

\end{proof}

\section{Linking the analytic and algebraic theories and forms which are not almost universal}\label{sec:link}
In this section, we prove Proposition \ref{prop:SiegelWeil} and Theorem \ref{thm:P14}, establishing that $P_{m}$ is not almost universal for all $m\equiv 2\pmod{12}$.  We draw on intuition from Proposition \ref{prop:spngen} to both motivate the proof of Theorem \ref{thm:P14} and explain the statement.  For $m\equiv 2\pmod{4}$ but $m\not\equiv 2\pmod{12}$, Proposition \ref{prop:spngen} (1) implies that there is only one spinor genus in the genus, and we proved in Theorem \ref{thm:main} that $P_m$ is indeed almost universal in this case.  On the other hand, for $m\equiv 2\pmod{12}$, Proposition \ref{prop:spngen} (2) implies that there are two spinor genera.  It is hence natural to search for ``primitive spinor exceptions'' for the lattice coset by studying whether there are families of exceptions in certain square classes; from the point of view of modular forms, we are searching for the component of the cuspidal part coming from unary theta functions, and Proposition \ref{prop:SiegelWeil} gives us a way to discover a unary theta function.
  
In order to prove Proposition \ref{prop:SiegelWeil}, we explicitly determine the genus and spinor genus of a lattice coset.  Let $L+\nu$ be the lattice coset associated to $P_{14}$ as in Section \ref{sec:algebraic}.  That is to say, $L:=\lan 12^2,12^2,12^2\ran$ and $\nu:=
\frac{5}{12}(e_1+e_2+e_3)$. 
\begin{lemma}\label{lem:genspnm=14}
\noindent

\noindent
\begin{enumerate}[leftmargin=*,label={\rm(\arabic*)}]
\item
 Defining 
\[
\mu:=\frac{1}{12}\!\left(5e_1+e_2+e_3\right),
\]
the classes in the genus of $L+\nu$ are then represented by $L+\nu$, $L+5\nu$, $L+\mu$, and $L+5\mu$.

\item
The cosets $L+\nu$ and $L+\mu$ form one spinor genus and the cosets $L+5\nu$ and $L+5\mu$ form the other spinor genus.  
\end{enumerate}
\end{lemma}
\begin{proof}

\noindent 

\noindent
(1) Suppose that $M+\mu'\in \gen(L+\nu)$.  The conductor, $c$, as defined in \cite{H14}, is the smallest positive integer for which $c\nu\in L$, or equivalently,  
\[
c=\prod_p[L_p+\Z_p[\nu]:L_p]=12.
\]
Since any local isometry from $L_p+\nu$ to $M_p+\mu'$ must send $M_p$ to $L_p$, the conductor is an invariant of the genus and hence $c$ is also the minimal positive integer for which $c\mu'\in M$.  Moreover, since $M_p\cong L_p$ at every prime $p$ it must follow that $M\cong L$ since $L$ has class number 1.  Therefore each class in the genus of $L+\nu$ contains a coset of the form $L+\mu'$, and it just remains to determine the possible values of $\mu'$.   From here, given that the conductor of the genus is $c=12$,
enumerating the possibilities yields a finite set of possibilities for the class representatives in the genus. Many of these classes are immediately seen not to be in the same genus as $L+\nu$ simply by comparing the numbers locally represented by these classes.  We further restrict the set by explicitly finding elements of $SO(V)$ between different cosets. From this, one concludes that the representatives for the classes are a subset of the four claimed lattice cosets. Furthermore, one easily checks that the theta functions of the four cosets are different (i.e., they each represent integers a different number of times), from which one concludes that they cannot be equivalent under the action of $SO(V)$. Finally, we construct $\sigma\in SO_{\mathbb{A}}(V)$ mapping each of the cosets to each other. 

For $p\neq 2,3$, we have $\nu,\mu\in L_p$ and hence 
\begin{align*}
L_p+5\nu&=L_p=L_p+\nu,\\
L_p+5\mu&=L_p=L_p+\mu,
\end{align*}
so the identity map suffices in this case.  When $p=2$, we observe that $4\nu=\frac{5}{3}(e_1+e_2+e_3)\in L_2$ and $4\mu=\frac{1}{3}(5e_1+e_2+e_3)\in L_2$, so that 
\begin{align*}
\nu&=-4\nu+5\nu\in L_2+5\nu,\\
\mu&= -4\mu+ 5\mu\in L_2+5\mu,
\end{align*}
and $\frac{5}{3}e_1\in L_2$ implies that
\[
\nu=\frac{5}{12}(e_1+e_2+e_3)=5\mu-\frac{5}{3}(e_1)\in L_2+5\mu,
\]
implying $L_2+5\mu = L_2+\mu = L_2+\nu=L_2+5\nu$.  When $p=3$, then we consider the symmetries $\tau_{e_i}$ of $L_3$ which negate the vector $e_i$, and the symmetry $\tau_{e_2-e_3}$ which switch $e_2$ and $e_3$ and fix everything else.  Then 
\[
\tau_{e_2}\circ \tau_{e_3}(\nu)=\frac{5}{12}(e_1-e_2-e_3)=\mu-\frac{1}{2}(e_2+e_3)
\]
so $L_3+\nu\cong L_3+\mu$ and the same isometry can be used to show that $L_3+5\nu\cong L_3+5\mu$.  

Similarly,
\[
\tau_{e_1}\circ \tau_{e_2-e_3}(\mu)=\tau_{e_1}(\mu)=\frac{1}{12}(-5e_1+e_2+e_3)=5\nu-\frac{5}{2}e_1-2e_2-2e_3,
\]
and hence $L_3+\mu$ is isometric to $L_3+5\nu$.  From here we conclude that $L+\nu$, $L+5\nu$, $L+\mu$ and $L+5\mu$ are in the same genus. 

\noindent
\vspace{0.05in}

\noindent
(2)  By Proposition \ref{prop:spngen} (2), there are precisely 2 spinor genera in the genus of $L+\nu$.  Now it only remains to find representatives for the classes in the two spinor genera.  To do this, we need only find a map $\sigma=(\sigma_2,\sigma_3,...,\sigma_p,...)\in O'_\mathbb{A}(V)$ for which $\sigma_p(L_p+\nu)=L_p+\mu$ at every prime $p$.  For primes away from $2$ and $3$ we have 
\[
L_p+\mu=L_p=L_p+\nu,
\]
and so we can let $\sigma_p$ be the identity map for $p\neq 2,3$.  Moreover, when $p=2$, then 
\[
\nu=\mu+\frac{1}{3}e_2+\frac{1}{3}e_3
\]
so in fact $L_2+\nu=L_2+\mu$, and hence $\sigma_2$ can also be taken to be the identity map.   When $p=3$ we consider the symmetries $\tau_{e_2}$ and $\tau_{e_3}$, then 
\[
\tau_{e_2}\circ \tau_{e_3}(\nu)=\frac{5}{12}(e_1-e_2-e_3)=\mu-\frac{1}{2}e_2-\frac{1}{2}e_3,  
\]
and therefore we let $\sigma_3=\tau_{e_2}\circ\tau_{e_3}$.  Then clearly $\sigma$ is in the kernel of the adelic spinor norm map, since $Q(e_2)=Q(e_3)$, and this map sends $L+\nu$ to $L+\mu$.  A similar argument can be used to show that $L+5\nu$ and $L+5\mu$ are representatives for the two classes in the spinor genus of $L+5\nu$.    
\end{proof}

\begin{proof}[Proof of Proposition \ref{prop:SiegelWeil}]
The number of automorphs of either $L+\nu$ or $L+5\nu$ is $6$, while the number of automorphs of either $L+\mu$ or $L+5\mu$ is $2$.  Thus we conclude by Lemma \ref{lem:genspnm=14} that 
\[
\Theta_{\spin(L+\nu)}= \frac{3}{2}\left(\frac{\Theta_{L+\nu}}{6} + \frac{\Theta_{L+\mu}}{2}\right)
\]
and 
\begin{equation}\label{eqn:Thetagenm=14}
\mathcal{E}_{(L+\nu)}=\Theta_{\genus(L+\nu)} = \frac{3}{4}\left( \frac{\Theta_{L+5\nu}}{6} + \frac{\Theta_{L+5\mu}}{2}+\frac{\Theta_{L+\nu}}{6} + \frac{\Theta_{L+\mu}}{2}\right).
\end{equation}
We claim that 
\begin{equation}\label{eqn:Thetaspnm=14}
\Theta_{\spin(L+\nu)}(\tau) = \Theta_{\genus(L+\nu)}(\tau) -\frac{1}{8}\vartheta_{1,3,12}(\tau),
\end{equation}
which would imply the claim.  Both sides are modular forms of weight $3/2$ on $\Gamma_1(4\cdot 12^2)$ with the usual $\Theta^3$-multiplier.

Recall now that by the valence formula, a modular form of weight $k$ for $\Gamma\subseteq\SL_2(\Z)$ with some multiplier is uniquely determined by the first 
\[
\frac{k}{12}[\SL_2(\Z):\Gamma]
\]
Fourier coefficients, where $[\SL_2(\Z):\Gamma]$ is the index of $\Gamma$ in $\SL_2(\Z)$.  Since (cf. \cite[p. 2]{OnoBook})
\[
\left[\SL_2(\Z):\Gamma_1(N)\right]=N^2\prod_{p\mid N}\left(1-\frac{1}{p^2}\right),
\]
we have 
\[
\left[\SL_2(\Z):\Gamma_1(24^2)\right]=24^4\left(1-\frac{1}{4}\right)\left(1-\frac{1}{9}\right)= 221184.
\]

Hence we only need to check $\frac{3}{24}\cdot 221184 =27648$ Fourier coefficients to verify \eqref{eqn:Thetaspnm=14}.  This is easily verified with a computer by computing the relevant theta series.

\end{proof}

In order to prove Theorem \ref{thm:P14}, we use the $m=14$ case as a springboard from which the other cases follow.  In particular, we show the following theorem.
\begin{theorem}\label{thm:XYZmod12}
\noindent

\noindent
\begin{enumerate}[leftmargin=*,label={\rm(\arabic*)}]
\item If $\ell\equiv 1\pmod{12}$ is an odd prime, then 
\[
X^2+Y^2+Z^2=3\ell^2
\]
has no solutions in $X,Y,Z\in\Z$ with $X\equiv Y\equiv Z\equiv 5\pmod{12}$.
\item If $\ell\equiv 7\pmod{12}$ is an odd prime, then 
\[
X^2+Y^2+Z^2=3\ell^2
\]
has no solutions in $X,Y,Z\in\Z$ with $X\equiv Y\equiv Z\equiv 1\pmod{12}$.
\end{enumerate}
\end{theorem}
\begin{proof}
\noindent

\noindent
(1)  As in the proof of Proposition \ref{prop:SiegelWeil} (and as in Section \ref{sec:algebraic}), we let $L+\nu$ be the lattice coset associated to $P_{14}$.  The claim is equivalent to the statement that $L+\nu$ does not represent $3\ell^2$ for all $\ell\equiv 1\pmod{12}$.  Since Conjecture \ref{conj:SiegelWeil} is true for the spinor genus of $L+\nu$ by Proposition \ref{prop:SiegelWeil} and the Fourier coefficients of each $\Theta_{M+\nu'}$ are non-negative, the $3\ell^2$-th coefficient of the theta function $\Theta_{\spin(L+\nu)}$ is zero if and only if the $3\ell^2$-th coefficient of $\Theta_{M+\nu'}$ is zero for all $M+\nu'\in \spn(L+\nu)$.  In particular, this implies that if the $3\ell^2$-th coefficient of $\Theta_{\spin(L+\nu)}$ always vanishes, then the claim is true. 

We next show that these coefficients of $\Theta_{\spin(L+\nu)}$ do indeed vanish. In order to show this, we explicitly compute the Eisenstein series $\Theta_{\genus(L+\nu)}=\mathcal{E}_{\genus(L+\nu)}$ and the linear combination of unary theta functions $\mathcal{U}_{\spin(L+\nu)}$.  By \eqref{eqn:Thetaspnm=14}, we have 
\[
\mathcal{U}_{\spin(L+\nu)}=-\frac{1}{8}\vartheta_{1,3,12}.
\]
We next use \eqref{eqn:Thetagenm=14} to compute the Eisenstein series component $\mathcal{E}_{\genus(L+\nu)}$.  To ease notation, we define the \begin{it}sieve operator\end{it} $S_{N,\ell}$, acting on Fourier expansions $f(\tau)=\sum_{n\geq 0 }a_f(n) q^n$ by
\[
f|S_{N,\ell}(\tau):=\sum_{\substack{n\geq 0\\ n\equiv \ell\pmod{N}}} a_f(n)q^n.
\]
Then a straightforward elementary calculation (by splitting the representations $x^2+y^2+z^2=24n+3$ via the congruence classes of $x$, $y$, and $z$) yields
\[
\Theta^3\big|S_{24,3}(\tau)= 48\left( \frac{\Theta_{L+5\nu}(\tau)}{6} + \frac{\Theta_{L+5\mu}(\tau)}{2}+\frac{\Theta_{L+\nu}(\tau)}{6} + \frac{\Theta_{L+\mu}(\tau)}{2}\right) + 8\Theta^3(3\tau)\big|S_{24,3}.
\]
Hence by \eqref{eqn:Thetagenm=14}, we have
\[
\mathcal{E}_{\genus(L+\nu)}=\frac{1}{64}\left(\Theta^3\big|S_{24,3}(\tau)-\frac{1}{8}\Theta^3(3\tau)\big|S_{24,3}\right)
\]
In particular, the $3\ell^2$-th coefficient of $\mathcal{E}_{\genus(L+\nu)}$ is exactly the number of ways to write $3\ell^2$ as the sum of $3$ squares.  By \cite[Theorem 86]{Jones}, since the quadratic form $Q(x,y,z)=x^2+y^2+z^2$ has class number $1$, this coefficient is given by 
\[
24H\!\left(3\ell^2\right),
\]
where 
\[
H(d):=\sum_{\substack{f\in\N\\ -\frac{d}{f^2}\equiv 0,1\pmod{4}}}\frac{h\!\left(-\frac{d}{f^2}\right)}{u\!\left(-\frac{d}{f^2}\right)}
\]
denotes the Hurwitz class number, with $h(D)$ denoting the usual class number and $u(D)$ being half the size of the automorphism group of the order of discriminant $D$ in $\Q(\sqrt{D})$.  

However, for $d=3\ell^2$, the class number formula \cite[Corollary 7.28, page 148]{Cox} and $h(-3)=1$ (as well as the fact that $u(-3)=3$ and $u(-3r^2)=1$ for $r>1$) imply that (for $\ell$ prime)
\begin{equation}\label{eqn:Heval}
H\!\left(3\ell^2\right) = \frac{h(-3)}{3} + h\!\left(-3\ell^2\right) = \frac{1}{3}+ \frac{1}{3}\left(\ell-\!\left(\frac{-3}{\ell}\right)\right)= \frac{1}{3}\!\left(\ell +1-\!\left(\frac{-3}{\ell}\right)\right).
\end{equation}
Here $(-3/\ell)$ is the Kronecker--Jacobi--Legendre symbol, which for $\ell\equiv 1\pmod{3}$ is $1$ in particular.  Thus for prime $\ell\equiv 1\pmod{3}$, the coefficient of $3\ell^2$ in $\mathcal{E}_{\genus(L+\nu)}$ is 
\begin{equation}\label{eqn:Eiscoeffm=14}
\frac{1}{64}\cdot 24 \cdot \frac{\ell}{3}= \frac{\ell}{8}.
\end{equation}
At the same time,
\begin{equation}\label{eqn:hrewrite}
\vartheta_{1,3,12}(\tau)= \sum_{\substack{n\in\Z\\ n\equiv 1\pmod{4}}} n q^{3n^2} = \sum_{n\geq 0} \left(\frac{-4}{n}\right) n q^{3n^2}.
\end{equation}
Thus the $3n^2$-th coefficient of $\vartheta_{1,3,12}(\tau)/8$ is $\left(\frac{-4}{n}\right) n/8$.  For $n=\ell\equiv 1\pmod{4}$, we specifically have $\ell/8$, which cancels with the coefficient from $\mathcal{E}_{\genus(L+\nu)}$.  Hence by Proposition \ref{prop:SiegelWeil} (in particular, see \eqref{eqn:Thetaspnm=14}), the $3\ell^2$-th coefficient of $\Theta_{\spin(L+\nu)}$ is zero.  This yields the claim.
\noindent
\vspace{0.05in}

\noindent
(2) We argue similarly to part (1), except this time we use the lattice coset $L+5\nu$ instead of $L+\nu$ (the statement is a rewording of the claim that $L+5\nu$ does not represent any integer of the form $3\ell^2$ with $\ell\equiv 7\pmod{12}$ prime).  The classes in the spinor genus of $L+5\nu$ are given by $L+5\nu$ and $L+5\mu$.  Moreover, by \eqref{eqn:Thetagenm=14} we have 
\[
\Theta_{\genus(L+5\nu)}=\Theta_{\genus(L+\nu)}= \frac{1}{2}\left(\Theta_{\spin(L+\nu)}+\Theta_{\spin(L+5\nu)}\right).  
\]
Rearranging and plugging in \eqref{eqn:Thetaspnm=14}, we have 
\[
\Theta_{\spin(L+5\nu)} = 2\Theta_{\genus(L+\nu)} -\Theta_{\spin(L+\nu)} \overset{\eqref{eqn:Thetaspnm=14}}{=} \Theta_{\genus(L+\nu)}+\frac{1}{8}\vartheta_{1,3,12}(\tau).  
\]
We then use \eqref{eqn:Heval} and \eqref{eqn:hrewrite} to compute the $3\ell^2$-th coefficient of each side.  For $\ell\equiv 7\pmod{12}$ prime, we have $(-3/\ell)=1$ so that \eqref{eqn:Eiscoeffm=14} yields that the $3\ell^2$-th coefficient of the Eisenstein series $\Theta_{\genus(L+\nu)}$ is precisely $\ell/8$.  Since $(-4/\ell)=-1$ for $\ell\equiv 7\pmod{12}$, the coefficient of $\vartheta_{1,3,12}(\tau)/8$ is $-\ell/8$, giving cancellation.  We conclude that the spinor genus of $L+5\nu$ does not represent $3\ell^2$ by Proposition \ref{prop:SiegelWeil}.
\end{proof}

We are now ready to prove Theorem \ref{thm:P14}.
\begin{proof}[Proof of Theorem \ref{thm:P14}]
For $m\equiv 2\pmod{12}$, we write $m=12r+2$.  We claim that, in particular, $P_m$ does not represent $n$ whenever 
\[
2(m-2)n + 3\left(\frac{m-4}{2}\right)^2 = 3\ell^2,
\]
where $\ell$ is any prime satisfying 
\begin{equation}\label{eqn:ellmod12}
 \begin{cases} \ell\equiv 1\pmod{12} & \text{if $r$ is odd,}\\
\ell\equiv 7\pmod{12} & \text{if $r$ is even.}
\end{cases}
\end{equation}
Note first that $P_m$ represents $n$ if and only if there exist $x,y,z\in\Z$ such that 
\begin{align}
\nonumber 24rn+3(6r-1)^2 &= 2(m-2)n+3\!\left(\frac{m-4}{2}\right)^2\\
\nonumber&=\!\left((m-2)x+\frac{m-4}{2}\right)^2\!\!+\!\left((m-2)y+\frac{m-4}{2}\right)^2\!\!+\!\left((m-2)z+\frac{m-4}{2}\right)^2\\
\label{eqn:badident}&= \left(12rx+6r-1\right)^2+\left(12ry+6r-1\right)^2+\left(12rz+6r-1\right)^2.
\end{align}
Notice that since $(6r-1)^2\equiv 1\pmod{24}$, the left hand side of \eqref{eqn:badident} is congruent to $3$ modulo $24$, so we may write it in the shape $24n'+3$ for some $n'$.  Writing $X:=12rx+6r-1$, $Y:=12ry+6r-1$, and $Z:=12rz+6r-1$, if \eqref{eqn:badident} holds, then there hence exist $X,Y,Z\in\Z$ with $X\equiv Y\equiv Z\equiv 6r-1 \pmod{12}$ for which $X^2+Y^2+Z^2=24n'+3$.  In particular, if $24n'+3=3\ell^2$ with $\ell$ a prime satisfying \eqref{eqn:ellmod12}, then Theorem \ref{thm:XYZmod12} implies that \eqref{eqn:badident} is not solvable.  

Therefore $P_m$ is not almost universal if there are infinitely many primes $\ell$ satisfying \eqref{eqn:ellmod12} for which $3\ell^2$ is in the set 
\[
\mathcal{S}:=\left\{ 2(m-2)n+3\!\left(\frac{m-4}{2}\right)^2: n\in\N_0\right\}=\left\{ 24rn+3(6r-1)^2: n\in\N_0\right\}.
\]
Hence, we need to find infinitely many $\ell$ satisfying \eqref{eqn:ellmod12} and $3\ell^2\equiv 3\pmod{24r}$, or in other words, we want  $\ell^2\equiv 1\pmod{8r}$ and $\ell\equiv 1\pmod{12}$ if $r$ is odd and $\ell\equiv 7\pmod{12}$ if $r$ is even.  For $r$ odd, we take $\ell\equiv 1\pmod{12r}$ sufficiently large.  For $r=2^a3^br'$ with $a>0$, we require $\ell\equiv 7\pmod{12}$ and $\ell^2\equiv 1\pmod{8r}$.  By the Chinese Remainder Theorem and Hensel's Lemma, there are infinitely many $\ell\equiv 1\pmod{r'}$, $\ell\equiv 1\pmod{3^{b+1}}$, and $\ell\equiv -1\pmod{4}$ such that $\ell^2\equiv 1\pmod{2^{3+a}}$, and these $\ell$ satisfy the desired congruences.   Therefore, there are infinitely many $\ell$ satisfying \eqref{eqn:ellmod12} for which $3\ell^2\in\mathcal{S}$ by the existence of infinitely many primes in arithmetic progressions.  Each such $\ell$ corresponds to some $n$ which is not represented by $P_m$, yielding the claim.

\end{proof}

\end{document}